\documentclass[11pt,reqno]{amsart}

\usepackage{epsfig}

\usepackage{graphicx}
\usepackage{amscd}
\usepackage{amsmath}
\usepackage{amsfonts}
\usepackage{yfonts}
\usepackage{psfrag}
\usepackage{amssymb}
\usepackage{verbatim}
\usepackage{comment}

\def\bpi{\mbox{\boldmath{$\pi$}}}
\def\btau{\mbox{\boldmath{$\tau$}}}

\def\bp{{\bf p}}
\newcommand{\Prob}{{\mathbb P}\,}

\def\Pmu{\Prob_{\!\!\mu}}
\def\Pdel{\Prob_{\!\delta}}

\newcommand{\diam}{\mbox{\rm diam}}

\newcommand{\be}{\begin{eqnarray}}
\newcommand{\ee}{\end{eqnarray}}

\def\wtil{\widetilde}

\newcommand{\half}{\frac{1}{2}}

\newcommand{\eps}{{\varepsilon}}

\newcommand{\Q}{{\mathbb Q}}
\newcommand{\Z}{{\mathbb Z}}
\newcommand{\Nat}{{\mathbb N}}
\def\N{\Nat}

\def\Lk{{\mathcal L}}
\def\Sk{{\mathcal S}}

\newcommand{\Hau}{{\mathcal H}}
\def\Hk{\Hau}

\def\Sig{\Sigma}

\newcommand{\Lam}{{\Lambda}}
\def\Gam{{\Gamma}}

\newcommand{\om}{\omega}
\newcommand{\gam}{\gamma}

\def\Exp{{\mathbb E}}

\newtheorem{theorem}{Theorem}[section]

\newtheorem{lemma}[theorem]{Lemma}
\newtheorem{corollary}[theorem]{Corollary}
\newtheorem{proposition}[theorem]{Proposition}

\theoremstyle{definition}

\theoremstyle{definition}

\numberwithin{equation}{section}

\begin{document}

\title{The Multiplicative golden mean shift has infinite Hausdorff measure}
\author{Yuval Peres}
\address{Yuval Peres, Microsoft Research, One Microsoft Way, Redmond, WA 98052, USA} \email{peres@microsoft.com}
\author{Boris Solomyak}
\address{Boris Solomyak, University of Washington, Box 354350, Dept. of Math., Seattle, WA 98195, USA}
\email{solomyak@math.washington.edu}

\begin{abstract}
In an earlier work, joint with R. Kenyon, we computed the Hausdorff dimension of the ``multiplicative golden mean shift'' defined as the set of
all reals in $[0,1]$ whose binary expansion $(x_k)$ satisfies $x_k x_{2k}=0$ for all $k\ge 1$. Here we show that this set has infinite
Hausdorff measure in its dimension. A more precise result in terms of gauges in which the Hausdorff measure is infinite is also obtained.
\end{abstract}

\maketitle

\section{Introduction}

\begin{sloppypar}
Consider the set
$$
\Xi_{G}:= \Bigl\{ x = \sum_{k=1}^\infty x_k 2^{-k}:\ x_k \in \{0,1\},\ x_k x_{2k}=0 \ \mbox{for all}\ k\Bigr\}
$$
which we call the ``multiplicative golden mean shift''. The reason for this term is that the
set of binary sequences corresponding
to the points of $\Xi_{G}$ is invariant under the action of the semigroup of multiplicative positive integers $\N^*$:
$
M_r(x_k) = (x_{rk})\ \ \mbox{for}\ r\in \N.
$
Fan, Liao, and Ma \cite{Fan} showed that
$
\dim_M(\Xi_{G}) = \sum_{k=1}^\infty
2^{-k-1}\log_2 F_{k+1}= 0.82429\ldots,
$
where $F_k$ is the $k$-th Fibonacci number: $F_1=1,\ F_2 = 2,\ F_{k+1} = F_{k-1}+F_k$, and raised the question of computing
the Hausdorff dimension of $\Xi_{G}$.
\end{sloppypar}

\begin{theorem}[\cite{KPS1,KPS2}] \label{prop-gold}
We have  $\dim_H(\Xi_{G}) < \dim_M(\Xi_{G})$. In fact,
\be \label{gold1}
\dim_H(\Xi_{G}) = -\log_2 p = 0.81137\ldots,\ \ \mbox{where}\ p^3=(1-p)^2,\ \ \ 0<p<1.
\ee
\end{theorem}

Here we prove

\begin{theorem} \label{th-gauge}
{\bf (i)} The set $\Xi_G$ has infinite (not $\sigma$-finite) Hausdorff measure in its dimension.
Moreover, let $s=\dim_H(\Xi_{G})$. Then
$\Hk^\phi(\Xi_G)=\infty$ for
\be \label{eq-gauge1}
\phi(t) = t^s \exp\Bigl[ -c \frac{|\log t|}{(\log|\log t|)^2}\Bigr]
\ee
provided that $c>0$ is sufficiently small, and furthermore, $\Xi_G$ is not $\sigma$-finite with respect to $\Hk^\phi$.

{\bf (ii)} On the other hand, we have $\Hk^{\psi_\theta}(\Xi_G)=0$ for
\be \label{eq-gauge2}
\psi_\theta(t) = t^s \exp\Bigl[ -\frac{|\log t|}{(\log|\log t|)^\theta}\Bigr],
\ee
provided that $\theta<2$.
\end{theorem}

\noindent
{\bf Remarks.} 1. In \cite{KPS2} we have pointed out a remarkable analogy between dimension properties of
multiplicative shifts of finite type and self-affine carpets of Bedford and McMullen,
see \cite{Bedf,McMullen}, although we are not aware of any
direct connection. The stated theorem provides further
evidence of this: it exactly corresponds to Theorem 3 from the paper by the first-named author \cite{Peres}.
We should point out, however, that our proof requires many new elements; in particular, the recurrence relation from Lemma~\ref{lem-poly} 
below has no
parallels in \cite{Peres}.

2. For self-affine carpets with non-uniform horizontal fibres, there is an elegant ``soft'' argument showing that the Hausdorff
measure of the set in its dimension cannot be positive and finite \cite{LG}, and more generally, this holds for any gauge
\cite{Peres}. It would be interesting to find a similar argument for
the multiplicative golden mean shift as well.

3. We expect that similar results hold for other multiplicative shifts of finite type considered in \cite{KPS2}.
Since the proofs are quite technical, we decided to focus on the most basic example of $\Xi_G$.


\section{Preliminaries and the scheme of the proof}

It is more convenient to work in the symbolic space $\Sig_2 = \{0,1\}^\N$, with the metric
$$
\varrho((x_k) ,(y_k)) = 2^{-\min\{n:\ x_n \ne y_n\}}.
$$
It is well-known that the dimensions of a compact subset of $[0,1]$ and the corresponding set of binary digit sequences in $\Sig_2$
are equal (this is equivalent to replacing the covers by arbitrary interval with
those by dyadic intervals), and the Hausdorff measures in the gauges that we are considering are comparable, up to
a multiplicative constant. Thus, it suffices to
work with the set $X_{G}$---the collection of all binary sequences $(x_k)$ such
that $x_k x_{2k}=0$ for all $k$.
Observe that
\be \label{eq1}
X_{G} = \Bigl\{\om =
{(x_k)}_{k=1}^\infty \in \Sig_2:\ {(x_{i2^r})}_{r=0}^\infty \in \Sig_{G}\ \ \mbox{for all}\ i \ \mbox{odd}\Bigr\}
\ee
where $\Sig_{G}$ is the usual (additive) golden mean shift:
$$
\Sig_{G}:= \{{(x_k)}_{k=1}^\infty\in \Sig_2,\ x_k x_{k+1}=0,\ \ \mbox{for all}\ k\ge 1\}.
$$

We will use the Rogers-Taylor density theorem from \cite{RT}.
We state it in the symbolic space $\Sig_2$ where $[u]$ denotes the cylinder set of sequences starting with
a finite ``word'' $u$ and $x_1^n = x_1\ldots x_n$.
Given a continuous increasing function $\phi$ on $[0,\infty)$, with $\phi(0)=0$, we consider the generalized Hausdorff measure
with the gauge $\phi$, denoted by $\Hk^\phi$, see e.g.\ \cite[p.33]{Falconer} or \cite[p.50]{Rogers} for the definition and basic
properties.

\begin{theorem}[Rogers and Taylor]\label{th-density}
Let $\Prob$ be a finite Borel measure on $\Sig_2$ and let $\Lam$ be a Borel set in $\Sig_2$
such that $\Prob(\Lam)>0$. Let $\phi$ be any gauge function. If for all $x\in \Lam$,
\be \label{eq-density}
\beta_1\le \liminf_{n\to \infty} \frac{\phi(2^{-n})}{\Prob[x_1^n]} \le \beta_2
\ee
(where $\beta_1,\beta_2$ may be zero or infinity),
then
$$
c_1\beta_1\Prob(\Lam) \le \Hk^\phi(\Lam) \le c_2\beta_2 \Prob(\Lam),
$$
where $c_1$ and $c_2$ are positive and finite.
\end{theorem}

\begin{corollary}\label{cor-dens}
Let $\Prob$ be a finite Borel measure on $\Sig_2$ and let $\Lam$ be a Borel set in $\Sig_2$
such that $\Prob(\Lam)>0$. Let $\phi$ be any gauge function. 

{\bf (i)} If for $\Prob$-a.e.\ $x\in \Lam$
$$
\lim_{n\to \infty} (\log_2\Prob[x_1^n]-\log_2\phi(2^{-n}))=-\infty,
$$
then $\Hk^\phi(\Lam) = \infty$.

{\bf (ii)} If for all $x\in \Lam$
$$
\lim_{n\to \infty} (\log_2\Prob[x_1^n]-\log_2\phi(2^{-n}))=+\infty,
$$
then $\Hk^\phi(\Lam) = 0$.
\end{corollary}

For an odd $i$ denote by $J(i) = \{2^r i\}_{r=0}^\infty$ the geometric progression with ratio 2 starting at $i$. Equation
(\ref{eq1}) says that $x\in X_G$ if and only if the ``restriction'' of $x$ to every $J(i)$ belongs to $\Sig_G$.
We can define a measure on $X_G$ by taking an infinite product of probability measures on each ``copy'' of $\Sig_G$.

In order to compute $\dim_H(X_G)$, it was enough to take the same measure $\mu$ on each copy, see \cite{KPS1}. Given a
probability measure $\mu$ on $\Sig_{G}$, we define a probability measure on $X_{G}$ by
\be \label{eq-meas1}
\Pmu[u]:= \prod_{i\le n,\, i\ \mbox{\tiny odd}} \mu[u|_{J(i)}],
\ee
where $u|_{J(i)}$ denotes the ``restriction'' of the word $u$ to the subsequence $J(i)$.
It was proved in \cite{KPS1,KPS2} that there is a unique probability measure $\mu$ on $\Sig_G$ such that
$\dim_H(\Pmu) = \dim_H(X_G)$. Denote by $\mu(r)$ the Markov
(non-stationary) measure on $\Sig_G$,  with initial probabilities $(r,1-r)$ and the matrix of transition probabilities
$P = (P(i,j))_{i,j=0,1} = \left( \begin{array}{cc} r & \ 1-r \\ 1 & 0 \end{array} \right)$.
Then $\mu=\mu(p)$, where $p^3 = (1-p)^2$. The measure $\mu(r)$ on cylinder sets can be explicitly written as follows:
\be \label{eq-Markov}
\mu(r)([u_1\ldots u_k])=(1-r)^{N_1(u_1\ldots u_k)} r^{N_0(u_1\ldots u_{k})-N_1(u_1\ldots u_{k-1})},
\ee
where $u\in \{0,1\}^k$ is a word admissible in $\Sig_G$, i.e.\ if $u_j=1$ then $u_{j+1}=0$ for $j\le k-1$, and
$N_i(u)$ denotes the number of symbols $i$ in the word $u$. To verify (\ref{eq-Markov}), note that
the probability of a 1 is always $1-r$ (including the first
position), and the probability of a 0 is $r$, except when it follows a 1, in which case its probability equals one.

For the lower bound, i.e.\ part (i) of Theorem~\ref{th-gauge},
we have to ``fine-tune'' the measure $\Pmu$ by taking a product of measures $\mu(p_k)$ on subsequences $J(i)$ with
odd $i$ such that $2^{k}\le i < 2^{k+1}$. It is clear that we must have $\lim_{k\to \infty}p_k = p$; in fact, we will take
$p_k = p+\frac{\delta}{k}$. More precisely, let
\be \label{defmuk}
\mu_k = \mu(p_k),\ \ \mbox{where}\ \ p_k = p + \frac{\delta}{k},\ k\ge 1,\ \ p_0=p,
\ee
and $\delta>0$ is sufficiently small, so that $p_1= p+\delta<1$.
Next, we define for $u\in \{0,1\}^n$, with $2^{\ell-1} < n \le 2^\ell$,
\be \label{eq-meas2}
\Pdel[u]:= \prod_{k=1}^{\ell} \ \ \ \prod_{\frac{n}{2^{k}} < i \le \frac{n}{2^{k-1}},\ i\ \mbox{\tiny odd}}
\mu_{\ell-k}[u|_{J(i)}],
\ee
where
$
u|_{J(i)}=u_i \ldots u_{2^{k-1} i}
$
is a word of length $k$. It is easy to see that $\Pdel$ is a probability measure on $X_G$.

Without loss of generality we can (and will) use logarithms base 2 in (\ref{eq-gauge1}) and (\ref{eq-gauge2}).
Theorem~\ref{th-gauge}(i) immediately follows from Corollary~\ref{cor-dens}(i) and the following proposition.

\begin{proposition} \label{prop-gauge}
There exist constants $\delta>0$ and $c>0$ such that the measure $\Pdel$ defined by (\ref{eq-meas2})
satisfies
$$
\lim_{n\to \infty} \left(\log_2 \Pdel[x_1^n] - \log_2 \phi(2^{-n})\right) = -\infty
$$
for $\Pdel$-a.e.\ $x\in X_G$,
where $\phi$ is the gauge function from (\ref{eq-gauge1}). Equivalently,
\be \label{eq-lim2}
\lim_{n\to \infty} \left(\log_2\Pdel[x_1^n] +ns + \frac{cn}{(\log_2 n)^2}\right) = -\infty
\ee
for $\Pdel$-a.e.\ $x\in X_G$, where $s=-\log_2 p=\dim_H(X_G)$.
\end{proposition}

For the upper bound of the Hausdorff measure, i.e.\ part (ii) of Theorem~\ref{th-gauge}, it is enough to take the same measure
$\mu=\mu(p)$ as in \cite{KPS1,KPS2}, however, the proof is rather delicate;
it follows the scheme of \cite[Theorem 3(ii)]{Peres}, but with many modifications.

We will need a classical large deviation inequality, which we state in the generality needed for us.

\begin{lemma}[Hoeffding's inequality \cite{Hoeffding}] \label{lem-ldev}
Let $\{X_i\}_{i\ge 1}$
be a sequence of independent random variables with expectation zero, such that $|X_i|\le C$, and let $S_n = \sum_{i=1}^n X_i$.
Then
\begin{equation} \label{eq-hoeffding}
\Prob\bigl( S_n \ge tn\bigr) \le \exp\Bigl(-\frac{t^2 n}{2C^2}\Bigr)
\end{equation}
for all $t>0$ and $n\ge 1$.
\end{lemma}


\section{Lower estimates of Hausdorff measure}

Here we prove Proposition~\ref{prop-gauge}.
We start with a reduction.

\begin{lemma} \label{lem-red}
If (\ref{eq-lim2}) holds for positive integers $n$ satisfying
\be \label{n-cond}
n=2^{\lfloor \ell/2 \rfloor} d,\ \ \mbox{where}\ \ 2^{\ell-1} < n \le 2^\ell,\ d\in \N,
\ee
with a constant $c>0$, then (\ref{eq-lim2}) holds for all $n$ with $c$ replaced by $c/2$.
\end{lemma}

\begin{proof} For a large integer $n\in (2^{\ell-1},2^\ell]$, let
$$
d:= \lfloor  2^{-\lfloor \ell/2 \rfloor}n \rfloor,\ \ \ m:= 2^{\lfloor \ell/2 \rfloor} d.
$$
Then
$$
n- \sqrt{2n} \le n - 2^{\lfloor \ell/2 \rfloor} < m \le n.
$$
It is clear that $m$ satisfies (\ref{n-cond}) (possibly with a different $\ell$). Observe that
\begin{eqnarray*}
\log_2\Pdel[x_1^n] + ns + \frac{(c/2)n}{(\log_2 n)^2} & \le & \log_2\Pdel[x_1^m] + ms + \frac{cm}{(\log_2 m)^2} + \\
                                                    & +   & s(n-m) + c\Bigl[ \frac{n/2}{(\log_2 n)^2} - \frac{m}{(\log_2 m)^2}\Bigr].
\end{eqnarray*}
Since
$$
s(n-m) + c\Bigl[ \frac{n/2}{(\log_2 n)^2} - \frac{m}{(\log_2 m)^2}\Bigr] \le s\sqrt{2n} +
c\Bigl[ \frac{n/2}{(\log_2 n)^2} - \frac{n-\sqrt{2n}}{(\log_2 n)^2}\Bigr]<0
$$
for large enough $n$, the claim follows.
\qed
\end{proof}

For $k\ge 1$ let
$\alpha_k$ be the partition of $\Sig_{G}$ into cylinders of length $k$.
For a measure $\mu$ on $\Sig_2$ and a finite partition $\alpha$, denote by $H^\mu(\alpha)$ the $\mu$-entropy of the
partition, with base $2$ logarithms:
$$
H^\mu(\alpha) = -\sum_{A\in \alpha} \mu(A)\log_2\mu(A).
$$

Let $n$ be such that (\ref{n-cond}) holds.
In view of (\ref{eq-meas2}),
\be \label{eq11}
\log_2 \Pdel[x_1^n]  \le \sum_{k=1}^{\lfloor \ell/2 \rfloor}\ \ \
\sum_{\frac{n}{2^{k}} < i \le \frac{n}{2^{k-1}},\ i\ \mbox{\tiny odd}} \log_2 \mu_{\ell-k}[x_1^n|_{J(i)}].
\ee
Note that $x_1^n|_{J(i)}$ is a word of length $k$ for $i\in (n/2^{k}, n/2^{k-1}]$, with $i$ odd, which is
a beginning of a sequence in $\Sig_G$.
Thus, $[x_1^n|_{J(i)}]$ is an element of the partition $\alpha_{k}$.
The random variables $x\mapsto \log_2\mu_{\ell-k} [x_1^n|_{J(i)}]$ are i.i.d\
for $i\in (n/2^{k}, n/2^{k-1}]$, with $i$ odd,
and their expectation  equals $-H^{\mu_{\ell-k}}(\alpha_k)$, by the definition of entropy.
Note that there are $n/2^{k+1}$ odds in $(n/2^k, n/2^{k-1}]$.
It is easy to see from (\ref{eq-Markov}) and (\ref{defmuk}) that 
\be \label{estim}
\bigl|\log_2\mu_{\ell-k} [x_1^n|_{J(i)}]\bigr|\le Ck,
\ee 
for $i\in (n/2^{k}, n/2^{k-1}]$, with some $C>0$,
independent of $n$ and $k$. Let
$$
S_{n/2^{k+1}}:= \sum_{\frac{n}{2^{k}} < i \le \frac{n}{2^{k-1}},\ i\ \mbox{\tiny odd}} \log_2 \mu_{\ell-k}[x_1^n|_{J(i)}]
$$
and
$S_{n/2^{k+1}}^*:= S_{n/2^{k+1}}+\frac{n}{2^{k+1}}H^{\mu_{\ell-k}}(\alpha_k),
$
be the corresponding sum of centered (zero expectation) random variables.
Then we have, for $k=1,\ldots \lfloor \ell/2 \rfloor$, and any $\eps\in (0,\half)$,
using (\ref{estim}) in Hoeffding's inequality (\ref{eq-hoeffding}):
\begin{eqnarray*} 
& & \Pdel\left(x:\ S_{n/2^{k+1}} > \frac{-n}{2^{k+1}} H^{\mu_{\ell-k}}(\alpha_k) + \Bigl(\frac{n}{2^{k+1}}\Bigr)^{1-\eps}\right) 
\\[1.2ex] & = & \Pdel\left(x:\ S_{n/2^{k+1}}^* > \Bigl(\frac{n}{2^{k+1}}\Bigr)^{1-\eps}\right)
\\[1.2ex]
& \le & \exp\Bigl[-\frac{(n/2^{k+1})^{1-2\eps}}{2C^2 k^2} \Bigr].
\end{eqnarray*}
Denote $b_\eps = \sum_{k=1}^\infty 2^{-(k+1)\eps}$. Now it follows from (\ref{eq11}) that
\begin{eqnarray} \label{eq-dev11}
& & \Pdel\left(x:\ \log_2\Pdel[x_1^n] > -n \sum_{k=1}^{\lfloor \ell/2 \rfloor} \frac{H^{\mu_{\ell-k}}(\alpha_k)}{2^{k+1}} + b_\eps n^{1-\eps}\right) 
\\[1.2ex]
&  \le & \Pdel\left(x:\ \sum_{k=1}^{\lfloor \ell/2 \rfloor} S_{n/2^{k+1}}^* > \sum_{k=1}^{\lfloor \ell/2 \rfloor} \Bigl(\frac{n}{2^{k+1}}\Bigr)^{1-\eps}
\right)
\nonumber \\[1.2ex]
& \le & \sum_{k=1}^{\lfloor \ell/2 \rfloor} \Pdel\left(x:\ S_{n/2^{k+1}}^* > \Bigl(\frac{n}{2^{k+1}}\Bigr)^{1-\eps}\right)
\nonumber \\[1.2ex]
& \le & \sum_{k=1}^{\lfloor \ell/2 \rfloor} \exp\Bigl[-\frac{n^{1-2\eps}}{2C^2 2^{(k+1)(1-2\eps)}k^2}\Bigr] \nonumber \\[1.2ex]
& \le & \ell \exp\Bigl[-\frac{C'}{\ell^2}\Bigl(\frac{n}{2^{\lfloor \ell/2 \rfloor+1}}\Bigr)^{1-2\eps}\Bigr] \nonumber \\[1.2ex]
& \le & \log_2(2n) \exp\Bigl[-\frac{C'}{\log^2_2(2n)}\Bigl(\frac{n}{8}\Bigr)^{\half-\eps}\Bigr] \label{eq-dev},
\end{eqnarray}
where we used that $\sqrt{8n} > 2^{1+\lfloor \ell/2 \rfloor}$ and
$\ell\le \log_2(2n)$ by (\ref{n-cond}) in the last step.
Since the last expression is summable in $n$, 
it follows from Borel-Cantelli that for $\Pdel$-a.e.\ $x\in X_G$, 
the event in parentheses in Equation (\ref{eq-dev11}) holds only for finitely many $n$. This
is the set of full $\Pdel$ measure for which we will prove (\ref{eq-lim2}), for $n$ satisfying (\ref{n-cond}).

Below we let $H(r) = -r\log_2 r - (1-r)\log_2(1-r)$.

\begin{lemma} \label{lem-poly}
We have, for any $r\in (0,1)$ and the measure $\mu(r)$ defined by (\ref{eq-Markov}),
\be \label{eq-entr}
H^{\mu(r)}(\alpha_k) = H(r)F_{k-1}(r),\ k\ge 1,
\ee
where $F_0(x)=1,\ F_1(x)=1+x$, and 
\be \label{eq-recur}
F_k(x) = 1 + xF_{k-1}(x) + (1-x)F_{k-2}(x),\ k\ge 2.
\ee
Moreover, the polynomials $F_k(x)$ can be expressed as follows:
\be \label{eq-poly}
F_k(x) = \frac{(x-1)^{k+2} - (k+2)x + (2k+3)}{(x-2)^2}\,,\ \ k\ge 0.
\ee
\end{lemma}

\begin{proof}
For $k=1$ the formula (\ref{eq-entr}) is trivially true. For $k\ge 2$ we have
$$
H^{\mu(r)}(\alpha_k) = H^{\mu(r)}(\alpha_1) + H^{\mu(r)}(\alpha_k|\alpha_1) = H(r) + H^{\mu(r)}(\alpha_k|\alpha_1).
$$
By the definition of conditional entropy and the properties of $\Sig_G$, we have
$$
H^{\mu(r)}(\alpha_k|\alpha_1) = r H^{\mu(r)}(\alpha_{k-1}) + (1-r)  H^{\mu(r)}(\alpha_{k-2}).
$$
(We set $H^{\mu(r)}(\alpha_0)=0$ here.) Indeed, 0 in $\Sig_G$ can be followed by an arbitrary element of 
$\Sig_G$, and 1 is followed by 0 and then by an arbitrary element of $\Sig_G$.
Now (\ref{eq-entr}) and (\ref{eq-recur}) are easily checked by induction. The explicit formula for $F_k(x)$ was found using
that 
$$
F_k(x)-F_{k-1}(x) = 1 - (1-x)(F_{k-1}(x)-F_{k-2}(x)),
$$
and can also be checked by induction.
\qed
\end{proof}

Since $\mu_{\ell-k} = \mu(p_{\ell-k})$, we have by (\ref{eq-entr}):
\be \label{eq12}
\sum_{k=1}^{\lfloor \ell/2 \rfloor} \frac{H^{\mu_{\ell-k}}(\alpha_k)}{2^{k+1}}=\sum_{k=1}^{\lfloor \ell/2 \rfloor}
 \frac{H(p_{\ell-k})F_{k-1}(p_{\ell-k})}{2^{k+1}}\,.
\ee
Recall that $p_{\ell-k}=p + \frac{\delta}{\ell-k}$.
Next we write the Taylor estimate at $p$, such that $p^3 = (1-p)^2$. We have $p\approx 0.56984>\half$, so it suffices to consider $x\in (\half,1)$. 
Below $C_i$ denote positive absolute constants.
It follows
from (\ref{eq-poly}) that 
\be \label{eq-asym1}
|F_k(x)| \le C_1k,\ \ |F'(x)|\le C_2 k,\ \ |F''(x)|\le C_3 k,\  x\in (1/2,1),\ k\ge 1.
\ee
Therefore,
\begin{eqnarray}
& & \left|\sum_{k=1}^{\lfloor \ell/2 \rfloor} \frac{H(p_{\ell-k})F_{k-1}(p_{\ell-k})}{2^{k+1}} - 
\sum_{k=1}^{\lfloor \ell/2 \rfloor} \frac{H(p)F_{k-1}(p)}{2^{k+1}} - 
 \sum_{k=1}^{\lfloor \ell/2 \rfloor} \frac{(HF_{k-1})'(p)}{2^{k+1}}\cdot \frac{\delta}{\ell-k}\right| \nonumber \\ 
& \le & C_4\sum_{k=1}^{\lfloor \ell/2 \rfloor} \frac{k}{2^{k+1}}\cdot\Bigl(\frac{\delta}{\ell-k}\Bigr)^2 \le C_5\frac{\delta^2}{\ell^2}\,. \label{eq13}
\end{eqnarray}

\begin{lemma} \label{lem-sum1}
We have
\be \label{eq-sum1}
\sum_{k=1}^\infty \frac{H(p)F_{k-1}(p)}{2^{k+1}} = s =-\log_2 p
\ee
and
\be \label{eq-sum2}
\sum_{k=1}^\infty \frac{(HF_{k-1})'(p)}{2^{k+1}} =0.
\ee
\end{lemma}

\begin{proof}
One can verify directly that $A(r):=H(r)\sum_{k=1}^\infty \frac{F_{k-1}(r)}{2^{k+1}} = \frac{2H(r)}{3-r}$, and this function achieves its maximum at $p$.
Alternatively, this follows from \cite{KPS1}, since $A(r)$ equals what was denoted $s(\mu)$ in  \cite{KPS1}, for $\mu=\mu(r)$. 
\qed
\end{proof}

In view of (\ref{eq-asym1}), we have $|\sum_{k={\lfloor \ell/2 \rfloor}+1}^\infty \frac{H(p)F_{k-1}(p)}{2^{k+1}}|\le 
C_6\ell\cdot 2^{-\ell/2}$, 
hence (\ref{eq-sum1}) implies
\be \label{eq14}
\left|\sum_{k=1}^{\lfloor \ell/2 \rfloor} \frac{H(p)F_{k-1}(p)}{2^{k+1}} - s\right| \le C_6 \ell\cdot 2^{-\ell/2}.
\ee
Next, writing $\frac{1}{\ell-k} = \frac{1}{\ell} + \frac{k}{\ell^2} + \frac{k^2}{\ell^2(\ell-k)}$, we obtain
\begin{eqnarray*}
& & \sum_{k=1}^{\lfloor \ell/2 \rfloor} \frac{(HF_{k-1})'(p)}{2^{k+1}}\cdot \frac{\delta}{\ell-k}\\ 
& = & \frac{\delta}{\ell} \sum_{k=1}^{\lfloor \ell/2 \rfloor} \frac{(HF_{k-1})'(p)}{2^{k+1}}
+ \frac{\delta}{\ell^2} \sum_{k=1}^{\lfloor \ell/2 \rfloor} \frac{k(HF_{k-1})'(p)}{2^{k+1}}
+  \frac{\delta}{\ell^2} \sum_{k=1}^{\lfloor \ell/2 \rfloor} \frac{k^2(HF_{k-1})'(p)}{2^{k+1}(\ell-k)} \\
& =: & S_1 + S_2 + S_3. 
\end{eqnarray*}
Using (\ref{eq-asym1}), by (\ref{eq-sum2}) we have
\be \label{eq15}
|S_1| \le \frac{\delta}{\ell} \left|\sum_{k={\lfloor \ell/2 \rfloor}+1}^\infty \frac{(HF_{k-1})'(p)}{2^{k+1}}\right|
\le C_7 \frac{\delta}{\ell}\cdot \frac{\ell}{2^{\ell/2}} = \frac{C_7\delta}{2^{\ell/2}}\,,
\ee
and 
\be \label{eq16}
|S_3| \le C_8 \,\frac{\delta}{\ell^3}.
\ee
Finally,
\be \label{eq17}
\left|S_2 - \frac{\delta}{\ell^2} \sum_{k=1}^\infty \frac{k(HF_{k-1})'(p)}{2^{k+1}}\right| \le C_9 \frac{\delta}{\ell^2}\cdot \frac{\ell^2}{2^{\ell/2}} = 
\frac{C_9\delta}{2^{\ell/2}}\,.
\ee

\begin{lemma} \label{lem-sum2} We have
$$
\tau:= \sum_{k=1}^\infty \frac{k(HF_{k-1})'(p)}{2^{k+1}}>0.
$$
\end{lemma}

The proof uses a (rigorous) numerical calculation, and we postpone it to the end of the section.
Combining (\ref{eq12}), (\ref{eq13}), (\ref{eq14}), (\ref{eq15}), (\ref{eq16}), and (\ref{eq17}), we obtain
\be \label{eq-sum3}
\left|\sum_{k=1}^{\lfloor \ell/2 \rfloor} \frac{H^{\mu_{\ell-k}}(\alpha_k)}{2^{k+1}} - s - \frac{\tau\delta}{\ell^2}\right|
\le \frac{C_5\delta^2}{\ell^2} + \frac{C_6\ell}{2^{\ell/2}}+ \frac{(C_7+C_9)\delta}{2^{\ell/2}} + \frac{C_8\delta}{\ell^3}\,.
\ee
Now we can conclude the proof of the proposition. Let $x\in X_G$ be such that for all $n$ sufficiently large, satisfying (\ref{n-cond}), we have
$$
\log_2\Pdel[x_1^n] \le -n \sum_{k=1}^{\lfloor \ell/2 \rfloor} \frac{H^{\mu_{\ell-k}}(\alpha_k)}{2^{k+1}} + b_\eps n^{1-\eps}.
$$
Recall that this holds for $\Pdel$-a.e.\ $x$ by (\ref{eq-dev}) and Borel-Cantelli Lemma. Then from (\ref{eq-sum3})
we obtain, keeping in mind that $n\in (2^{\ell-1},2^\ell]$:
\begin{eqnarray*}
\Sk_n(x)
&:= & \log_2\Pdel[x_1^n] + ns + \frac{cn}{(\log_2 n)^2}  \\
& \le & \frac{cn}{(\log_2 n)^2} -\tau \delta \frac{n}{(\log_2 n)^2} + b_\eps n^{1-\eps} + \frac{C_5\delta^2 n}{(\log_2 n)^2} \\
& +& C_6\sqrt{n}\log_2 n + (C_7+C_9)\delta \sqrt{n} + \frac{C_8\delta n}{(\log_2 n)^3}.
\end{eqnarray*}
Now we choose a positive $\delta< \frac{\tau}{3C_5}$, which is possible by Lemma~\ref{lem-sum2},
so that $C_5\frac{\delta^2 n}{(\log_2 n)^2}< \frac{1}{3}  \frac{\tau \delta n}{(\log_2 n)^2}$, 
and then choose
$c\in (0,\tau \delta/3)$, whence
$$
\frac{cn}{(\log_2 n)^2}< \frac{1}{3}\cdot \tau \delta \frac{n}{(\log_2 n)^2}.
$$
Then
$$
\Sk_n(x) \le -\frac{1}{3}  \frac{\tau \delta n}{(\log_2 n)^2} + b_\eps n^{1-\eps} 
+C_6\sqrt{n}\log_2 n + (C_7+C_9)\delta \sqrt{n} + \frac{C_8\delta n}{(\log_2 n)^3} \to -\infty,
$$
as $n\to\infty$,
and (\ref{eq-lim2}) follows.
\qed

\medskip

\noindent {\em Proof of Theorem~\ref{th-gauge}(i).} As already mentioned, $\Hk^\phi(\Xi_G) = \Hk^\phi(X_G)=\infty$ follows from the Rogers-Taylor density
theorem (more precisely, from Corollary~\ref{cor-dens}(i)).
If $\Hk^\phi|_{\Xi_G}$ was $\sigma$-finite for some $c>0$, we would have $\Hk^\phi(\Xi_G)=0$ for all larger values of $c$, which is a contradiction.
\qed

\medskip

\noindent {\bf Remark.}
It is clear, without any calculation, that there exists $\gam>0$, arbitrarily small, such that
$$
\tau_\gam:=\sum_{k=1}^\infty \frac{k^{1+\gam} (HF_{k-1})'(p)}{2^{k+1}} \ne 0.
$$
This implies, by a minor modification of the argument, that $\Hk^{\phi_\gam}(X_G)=\infty$ for the gauge function
$$
\phi_\gam(t) = t^s \exp\Bigl[ -c \frac{|\log t|}{(\log|\log t|)^{2+\gam}}\Bigr].
$$
To this end, we need to take $p_k = p \pm \frac{\delta}{k^{1+\gam}}$ in (\ref{defmuk}), where the sign is that of $\tau_\gam$.
The details are left to the reader.

\medskip

\begin{proof}[Proof of Lemma~\ref{lem-sum2}]
A numerical calculation (we used {\em Mathematica}) showed that
$$
\sum_{k=1}^{12} \frac{k (HF_{k-1})'(p)}{2^{k+1}} \approx 0.187469.
$$
Thus, we only need to estimate the remainder.

We have $(HF_{k-1})'(p) = H(p)F_{k-1}'(p) + H'(p) F_{k-1}(p)$.
Recall that $p\approx 0.56984$, and a calculation gives
$$
H(p) \approx 0.68336 < 0.7,\ \ \ \ H'(p) \approx - 0.281198,\ \ \mbox{hence}\ \ |H'(p)| < 0.3.
$$
Recall (\ref{eq-poly}) that
$
F_n(x) = (x-2)^{-2} [(x-1)^{n+2} - (n+2)x + (2n+3)],
$
whence
$$
0 < F_n(p) < 2^{-(n+2)} - (n+2)/2 + (2n+3) < 3+3n/2.
$$
Further,
$$
F_n'(p) = \frac{2((p-1)^{n+2} - (n+2)p + (2n+3))}{(p-2)^3} + \frac{(n+2)(p-1)^{n+1} - (n+2)}{(p-2)^2}\,.
$$
Note that in the expression for $F'_n(p)$ the 1st term is positive and the 2nd term is negative. The first term, in absolute value, is less than
$2(3+3n/2) = 3n+6$, and the second term, in absolute value, is less than $n+3$ for $n\ge 1$. Thus,
$$
|F'_n(p)| < 3n+6.
$$
It follows (using a crude estimate) that
$$
\Bigl|\sum_{k=13}^\infty \frac{k(HF_{k-1})'(p)}{2^{k+1}} \Bigr| < \sum_{k=13}^\infty \frac{(0.7(3k+3) + 0.3(3k+3)/2)k}{2^{k+1}}<
 \sum_{k=13}^\infty \frac{3k(k+1)}{2^{k+1}}\,.
$$
Finally,
\begin{eqnarray*}
\sum_{k=13}^\infty \frac{3k(k+1)}{2^{k+1}} & = & (3/4) [(1-x)^{-1}x^{15}]''|_{x=1/2}
\\ &  = & (3/4)[2^{-11} + 30\cdot 2^{-12} + 15\cdot 14\cdot 2^{-12}] < 0.1 < 0.187469,
\end{eqnarray*}
completing the proof of the claim that $\sum_{k=1}^\infty \frac{k(HF_{k-1})'(p)}{2^{k+1}}>0$.
\qed \end{proof}


\section{Upper bound for Hausdorff measure}

First we give a short proof of a weaker result: $\Hk^\psi(X_G)=0$ where
\be \label{eq-gauge3}
\psi(t) = t^s \exp\Bigl[ -\frac{|\log_2 t|}{g(\log_2|\log_2 t|)}\Bigr]
\ee
where $g$ is increasing and $\int^\infty \frac{dt}{g(t)}=\infty$; in particular, this includes $\psi=\psi_\theta$ from (\ref{eq-gauge2}) with $\theta=1$.

\begin{proof}
We use the measure $\Pmu$ from (\ref{eq-meas1}), where $\mu=\mu(p)$, $p^3=(1-p)^2$, as in \cite{KPS1}. Consider {\bf any} point
$x\in X_G$. Then we obtain from
(\ref{eq-meas2}), as in \cite{KPS1}, for $n$ even:
\begin{eqnarray} \label{eq-pmu}
\Pmu[x_1^n] & = & (1-p)^{N_1(x_1^n)} p^{N_0(x_1^n)-N_1(x_1^{n/2})} \nonumber \\
            & = & p^n p^{N_0(x_1^{n/2})-N_0(x_1^n)/2},
\end{eqnarray}
in view of $1-p=p^{3/2},\ N_1(x_1^n)=n-N_0(x_1^n)$.
Note that $\log_2 \psi(2^{-n}) = -ns - \frac{n}{(\ln 2)g(\log_2 n)}$. In view of $s = -\log_2 p$, we have
\be \label{eq-upper}
\frac{\log_2 \Pmu[x_1^n]-\log_2 \psi(2^{-n})}{n} =
\frac{s}{2} \Bigl(\frac{N_0(x_1^{n/2})}{n/2} - \frac{N_0(x_1^n)}{n} \Bigr) + \frac{1}{(\ln 2) g(\log_2 n)}\,.
\ee
Denote
$$
b_j:= \frac{\log_2 \Pmu[x_1^{2^j}] - \log_2 \psi(2^{-2^j})}{2^j} = \frac{s}{2}\Bigl(\frac{N_0(x_1^{2^{j-1}})}{2^{j-1}} -
\frac{N_0(x_1^{2^j})}{2^j} \Bigr) + \frac{1}{(\ln 2) g(j)}\,.
$$
Then
$$
b_1 + \cdots + b_\ell =
\frac{s}{2} \Bigl(N_0(x_1^1)- \frac{N_0(x_1^{2^\ell})}{2^\ell}\Bigr) + \sum_{j=1}^\ell \frac{1}{(\ln 2) g(j)}\to +\infty,\ \ell\to \infty,
$$
by the assumption on the function $g$.
It follows that $\limsup 2^j b_j = +\infty$, hence
$$
\limsup_{n\to \infty} (\log_2 \Pmu[x_1^n]-\log_2 \psi(2^{-n})) = +\infty,
$$
and we obtain $\Hk^\psi(X_G) = 0$ by 
Corollary~\ref{cor-dens}(ii).
\qed \end{proof}

\medskip
                                          
Obtaining the same result for $\psi_\theta$ from (\ref{eq-gauge2}) with $1< \theta<2$ is more delicate. Our proof follows the scheme of the proof of
\cite[Theorem 3(ii)]{Peres}, but we have to make a number of modifications. The following lemma is a version of \cite[Lemma 5]{Peres}
in the form convenient for us.

\begin{lemma} \label{lem-seq}
{\bf (i)} Let $1<\eta<2$. Suppose that $\{\gam(n)\}_{n=1}^\infty$ is a real sequence such that
\begin{equation} \label{eq00}
C_1:=\sup_n |\gam(n)-\gam(n-1)| < \infty
\end{equation}
and for all $n\ge n_0$,
\begin{equation} \label{eq01}
\gam(n) \ge \frac{\gam(2n)}{2} + \frac{n}{(\log_2(2n))^\eta}\,.
\end{equation}
Then either there exists $c>0$ such that for all $n\ge n_0$
\be \label{eq0111}
\gam(2n) \ge c\frac{2n}{(\log_2 (2n))^{\eta-1}}\,,
\ee
or there exists $\eps>0$ such that for infinitely many $n$,
\begin{equation} \label{eq012}
\gam(2n) \le -\eps n\ \ \ \mbox{and}\ \ \ \gam(n) - \frac{\gam(2n)}{2} \le \frac{n}{\log_2(2n)}\,.
\end{equation}

{\bf (ii)} For any real sequence $\{\gam(n)\}_{n=1}^\infty$ satisfying (\ref{eq00}),
\begin{equation} \label{eq0121}
\gam(n) - \frac{\gam(2n)}{2} < \frac{n}{\log_2(2n)}
\end{equation}
for infinitely many $n$.
\end{lemma}

\begin{proof}
(i) Iterating (\ref{eq01}) we obtain for $n\ge n_0$ and $m\ge 1$:
\begin{equation} \label{eq02}
\gam(n) \ge \frac{\gam(2^m n)}{2^m} + n \cdot\sum_{j=1}^m \frac{1}{(j + \log_2 n)^\eta}\,.
\end{equation}

\noindent
\underline{Case 1:} $\gam(n)\ge 0$ for all $n\ge n_0$. Then (\ref{eq02}) implies for $n\ge n_0$:
$$
\gam(n) \ge n \sum_{j=1}^\infty \frac{1}{(j + \log_2 n)^\eta} \ge c \frac{n}{(\log_2 n)^{\eta-1}}\,,
$$
whence (\ref{eq0111}) holds.

\noindent
\underline{Case 2:} there exists $n_1\ge n_0$ such that $\gam(n_1) < -\eps<0$. Then (\ref{eq01}) implies $\gam(2n_1) \le 2\gam(n_1)
< -2\eps$, and
inductively, $\gam(2^m n_1) \le -2^m \eps$ for all $m\ge 1$. Moreover, for infinitely many $m$ we have
$$
\gam(2^{m-1}n_1) - \frac{\gam(2^m n_1)}{2} \le \frac{2^{m-1}n_1}{\log_2 (2^m n_1)}\,,
$$
since otherwise,
$$
\frac{\gam(2^{m-1}n_1)}{2^{m-1}} - \frac{\gam(2^m n_1)}{2^m} > \frac{n_1}{m + \log n_1},\ \ m\ge m_0+1,
$$
and then taking the sum over $m$ from $m_0+1$ to $\ell$ yields
$$
\frac{\gam(2^{m_0}n_1)}{2^{m_0}} - \frac{\gam(2^{m_0+\ell}n_1)}{2^{m_0+\ell}} \to \infty,\ \ \ell\to \infty,
$$
which is a contradiction, since $|\gam(i)|\le C_1i$ by (\ref{eq00}).
Thus, (\ref{eq012}) holds for infinitely many $n=2^m n_1$, as desired.

(ii) If the claim is not true, then (\ref{eq01}) holds for $n\ge n_0$ with $\eta=1$, for some $n_0\in \N$. 
Then we obtain (\ref{eq02}) with $\eta=1$. But $\gam(2^mn) \ge -C_12^mn$ by
(\ref{eq00}), and we get a contradiction letting $m\to \infty$.
\qed\end{proof}

\medskip

We still use the measure $\Pmu$ from (\ref{eq-meas1}), as in \cite{KPS1}, so by (\ref{eq-pmu}), keeping in mind that
$s=-\log_2p$, we have
\begin{equation} \label{equ1}
\log_2\Pmu[x_1^{2n}] + s(2n) = \bigl[N_0(x_1^{2n})/2-N_0(x_1^n)\bigr]s.
\end{equation}
Observe that
\begin{equation} \label{equ2}
N_0(x_1^n) = \sum_{k=1}^{\ell+1} \sum_{\frac{n}{2^{k}} < i \le \frac{n}{2^{k-1}},\ i\ \mbox{\tiny odd}} N_0(x_1^n|_{J_i}),\ \ \ 
2^{\ell-1} < n \le 2^\ell.
\end{equation}
By the definition of the measure $\Pmu$, the random variables $N_0(x_1^n|_{J_i})$ are i.i.d.\ for odd 
$i\in (\frac{n}{2^k},\frac{n}{2^{k-1}}]$. Note that $|x_1^n|_{J_i}|=k$ for such $i$, and the distribution of these random variables is
the distribution of $N_0(u)$, $|u|=k$, where $\{u_i\}$ is the Markov chain corresponding to $\mu$. By the definition of $\mu=\mu(p)$,
$$
\Exp\bigl[N_0[u]\bigr] = \sum_{j=0}^{k-1} (\bp P^j)_0,\ \ |u|=k,
$$
where $\bp = (p, 1-p)$ and $P = \left( \begin{array}{cc} p & 1-p \\ 1 & 0 \end{array} \right)$. Since $P$ has left eigenvectors
$\bpi = (\frac{1}{2-p}, \frac{1-p}{2-p})$ and $\btau = (1,-1)$ corresponding to the eigenvalues 1 and $p-1$, respectively, we have
$$
(\bp P^j)_0 = \frac{1}{2-p}\, [1- (p-1)^{j+2}],\ j\ge 0,
$$
hence
\begin{equation} \label{eq-exp}
\Exp \bigl[N_0[u]\bigr] = \frac{k}{2-p} - \frac{1}{2-p}\, \sum_{j=0}^{k-1} (p-1)^{j+2} = \frac{k}{2-p} - \frac{(1-(p-1)^k)(p-1)^2}{(2-p)^2}=:L_k.
\end{equation}

\begin{lemma} \label{lem-exp}
We have 
$$
\left|\frac{\Exp\bigl[N_0(x_1^{2n})\bigr]}{2} - \Exp\bigl[N_0(x_1^{n})\bigr]\right| \le C(\log_2 n)^2,\ \ n\in \N,
$$
for some $C>0$, where $x$ has the law of $\Pmu$.
\end{lemma}

\begin{proof}
Denote by $\Z_{\rm odd}(a,b]$ the set of odd integers in the interval $(a,b]$, where $a<b$ are reals. 
We have from (\ref{equ2}) and (\ref{eq-exp}):
$$
\Exp\bigl[N_0(x_1^{n})\bigr] = \sum_{k=1}^{\ell+1} \#\Z_{\rm odd}(\textstyle{\frac{n}{2^k}, \frac{n}{2^{k-1}}}]\cdot L_k.
$$
Note that $\Z_{\rm odd}(\textstyle{\frac{n}{2^{\ell+1}}, \frac{n}{2^{\ell}}}]=\{1\}$ 
if $n=2^\ell$, and it is empty otherwise. It follows that
\begin{equation} \label{equ3}
\frac{\Exp\bigl[N_0(x_1^{2n})\bigr]}{2} - \Exp\bigl[N_0(x_1^{n})\bigr] = \sum_{k=1}^{\ell+1} 
\Bigl(\frac{\#\Z_{\rm odd}(\textstyle{\frac{n}{2^{k-1}}, \frac{n}{2^{k-2}}}]}{2} - 
\#\Z_{\rm odd}(\textstyle{\frac{n}{2^k}, \frac{n}{2^{k-1}}}]\Bigr)
\cdot L_k + d\cdot L_{\ell+2},
\end{equation}
where $d\in \{0,1/2\}$. It is easy to see that 
\begin{equation} \label{eq-easy}
\left| \#\Z_{\rm odd}(a,b] - \Bigl(\frac{b-a}{2}\Bigr)\right| \le 1,\ \ \ \mbox{for}\ \ 0<a<b,
\end{equation}
hence, taking (\ref{eq-exp}) into account, 
$$
\left|\frac{\Exp\bigl[N_0(x_1^{2n})\bigr]}{2} - \Exp\bigl[N_0(x_1^{n})\bigr]\right| \le 2\sum_{k=1}^{\ell+2} L_k \le C'\sum_{k=1}^{\ell+2} k \le
C''\ell^2\le C (\log_2 n)^2.
$$
\qed\end{proof}

\medskip

\begin{proof}[Proof of Theorem \ref{th-gauge}(ii)]
In order to show that $\Hk^{\psi_\theta}(X_G)=0$, we cover $X_G$ by three subsets: $B$, $L$, and $\Lam$, defined as follows. Let
\begin{equation} \label{eq-defB}
B:= \left\{x\in X_G:\exists\,\eta>\theta,\ \ \frac{N_0(x_1^{2n})}{2} - N_0(x_1^n) >  \frac{-2n}{(\log_2(2n))^\eta}\ 
\mbox{for infinitely many $n$}\right\}.
\end{equation}
Denote
$$
N_0^*(x_1^n) := N_0(x_1^n) - \Exp\bigl[N_0(x_1^{n})\bigr]
$$
and let
\begin{eqnarray} \label{eq-defL}
L:= \Bigl\{x\in X_G:\ \exists\,\eps>0, &  & \frac{N_0(x_1^{2n})}{2} - N_0(x_1^n) \ge  \frac{-2n}{\log_2 (2n)}\ \ \ \mbox{and} \nonumber \\
                                       &  & 
N_0^*(x_1^{2n}) \le -\eps n\ \ \ \mbox{for infinitely many $n$}\Bigr\}. 
\end{eqnarray}
Finally, let $\Lam=X_G\setminus (L\cup B)$.
It suffices to verify that each of the three sets $B,L,\Lam$ has zero $\Hk^{\psi_\theta}$-measure (indeed, $L$ and $\Lam$
even have zero $\Hk^s$-measure).

\medskip

\noindent \underline{Step 1:} $\Hk^{\psi_\theta}(B)=0$. 
Let $B_\eta$ be the set of $x\in X_G$ such that the condition in (\ref{eq-defB}) holds for a fixed $\eta$. Thus, $B=\bigcup_{\eta>\theta} B_\eta=
\bigcup_{\eta\in \Q,\ \eta>\theta} B_\eta$, and it is enough to show that $\Hk^{\psi_\theta}(B_\eta)=0$. 
We have from (\ref{equ1}) and the definition of $\psi_\theta$ for all $x\in B_\eta$:
\begin{eqnarray*}
\log_2\Pmu[x_1^{2n}]-\log_2\psi_\theta(2^{-2n}) & = & s\Bigl(\frac{N_0(x_1^{2n})}{2}-N_0(x_1^n)\Bigr) + \frac{2n}{(\ln 2)(\log_2 (2n))^\theta} \\
                                    & > & \frac{-2ns}{(\log_2 (2n))^\eta} + \frac{2n}{(\ln 2)(\log_2 (2n))^\theta}
\end{eqnarray*}
for infinitely many $n$. Since $\eta>\theta$, it follows that
$$
\limsup_{n\to\infty} \bigl(\log_2\Pmu[x_1^{2n}]-\log_2\psi_\theta(2^{-2n})\bigr)=+\infty,
$$
hence $\Hk^{\psi_\theta}(B_\eta)=0$ by Theorem~\ref{th-density}.

\medskip

\noindent \underline{Step 2:} $\Hk^{\psi_\theta}(L)=0$. Denote by $L(\eps)$ the set of points $x\in X_G$ which satisfy the condition
in (\ref{eq-defL}) for a given $\eps>0$. For $\eps>0$ and $n\in \N$ let $\Lk_n(\eps)$ be the set of words $u$ of length $2n$ for which 
the condition in (\ref{eq-defL}) holds. (Note that this condition depends only on the first $2n$ symbols of $x$; thus, (\ref{eq-defL}) holds
for all $x\in [u]$.) If $u\in \Lk_n(\eps)$ then by  (\ref{equ1}) and (\ref{eq-defL}),
$$
\log_2(\Pmu[u]\, 2^{2ns}) \ge \frac{-2sn}{\log_2(2n)}\,,
$$
hence
\begin{equation} \label{eq111}
2^{-2ns} \le \exp\left(\frac{2sn(\ln 2)}{\log_2(2n)}\right) \Pmu[u].
\end{equation}
By the definition of $\Lk_n(\eps)$ 
we have
\begin{equation} \label{eq1111}
\sum_{u\in \Lk_n(\eps)} \Pmu[u]\le \Pmu\bigl(x:\ N_0^*(x_1^{2n})\le -\eps n\bigr).
\end{equation}

The following lemma is a consequence of large deviation estimates; it will be used in the last step of the proof as well.
\begin{lemma}\label{lem-ldev2}
There exist $c_2,c_3>0$ such that for all $t>0$ and $n\in \N$,
$$
\Pmu\bigl(x:\ |N_0^*(x_1^{2n})|\ge t n\bigr)\le c_2\exp(-c_3t^2 n).
$$
\end{lemma}

\begin{proof}
We have by (\ref{equ2}),
$$
N_0^*(x_1^{2n}) = \sum_{k=1}^{\ell+1} S^*_{A_k},\ \ \mbox{where}\ \ A_k = 
\#\Z_{\rm odd}(\textstyle{\frac{n}{2^k}, \frac{n}{2^{k-1}}}],
$$
$$
S^*_{A_k}:= \sum_{\frac{n}{2^{k}} < i \le \frac{n}{2^{k-1}},\ i\ \mbox{\tiny odd}} N_0^*(x_1^n|_{J_i}),
$$
and
$$
N_0^*(x_1^n|_{J_i}) = N_0(x_1^n|_{J_i}) - \Exp [N_0(u)] \ \ \mbox{for}\ \ |u|=k\ \ \mbox{and}\ \ 
i\in \Z_{\rm odd}(\textstyle{\frac{n}{2^k}, \frac{n}{2^{k-1}}}]. 
$$
Now,
$$
\Pmu\Bigl( \Bigl|\sum_{k=1}^{\ell+1} S^*_{A_k}\Bigr| \ge t n\Bigr) \le \sum_{k=1}^{\ell+1} \Pmu\Bigl( \Bigl|S^*_{A_k}\Bigr| \ge
\frac{t n}{k(k+1)}\Bigr),
$$
since $\sum_{k=1}^\infty \frac{1}{k(k+1)}=1$.
Note that $S^*_{A_k}$ is a sum of $A_k$ independent random variables, which are bounded by $k$ in modulus, hence by Hoeffding's inequality 
(\ref{eq-hoeffding}),
\begin{eqnarray*}
\Pmu\Bigl( \Bigl|S^*_{A_k}\Bigr| \ge \frac{t n}{k(k+1)}\Bigr) & = & \Pmu\Bigl(| S^*_{A_k}|\ge A_k\cdot \frac{t n}{k(k+1)A_k}\Bigr)\\
                                                    & \le & 2\exp\Big[-\frac{t^2 n^2}{2k^4(k+1)^2 A_k}\Bigr].
\end{eqnarray*}
Observe that $A_k \le \frac{n}{2^{k+1}}+1\le \frac{n}{2^k}$ by (\ref{eq-easy}), hence
$$
\Pmu\Bigl( \Bigl|S^*_{A_k}\Bigr| \ge \frac{t n}{k(k+1)}\Bigr)\le 2\exp\Big[-\frac{t^2 n\cdot 2^k}{2k^4(k+1)^2}\Bigr],
$$
and, therefore,
$$
\Pmu\bigl(x:\ |N_0^*(x_1^{2n})|\ge t n\bigr)\le 2\sum_{k=1}^\infty \exp\Big[-\frac{t^2 n\cdot 2^k}{2k^4(k+1)^2}\Bigr]\le
c_2\exp(-c_3t^2 n),
$$
for some positive $c_2, c_3$, as desired.
\qed\end{proof}

Combining (\ref{eq111}), (\ref{eq1111}) and Lemma~\ref{lem-ldev2}, with $t=\eps$, yields
$$
\sum_{u\in \Lk_n(\eps)} (2^{-2n})^s \le c_2\exp\Bigl[\frac{2sn(\ln 2)}{\log_2(2n)} -c_3\eps^2 n\Bigr].
$$
The right-hand side of this inequality  is summable in $n$, so by choosing large $n_0$ we can make the sum
$$
\sum_{n\ge n_0} \sum_{u\in \Lk_n(\eps)} (2^{-2n})^s = \sum_{n\ge n_0} \sum_{u\in \Lk_n(\eps)} (\diam[u])^s
$$
arbitrarily small. But for any $n_0$, the union $\bigcup_{n=n_0}^\infty \bigcup_{u\in \Lk_n(\eps)} [u]$ forms a cover of $L(\eps)$, proving that
$\Hk^s(L(\eps))=0$. Finally, $L = \bigcup_{\eps\in \Q} L(\eps)$, so we
obtain that $\Hk^s(L)=0$ and certainly $\Hk^{\psi_\theta}(L)=0$.

\medskip

\noindent \underline{Step 3.} For $\eta\in (1,2)$, $\eps\in (0,\eta)$, and $c>0$, let $\Lam(\eta,\eps,c)$ be the set of $x\in X_G$ such that for
$n$ sufficiently large we have
\begin{equation} \label{eq112}
\frac{N_0(x_1^{2n})}{2} - N_0(x_1^n) \le \frac{-2n}{(\log_2(2n))^\eta}
\end{equation}
and
\begin{equation} \label{eq113}
N_0^*(x_1^{2n}) \ge c\frac{2n}{(\log_2(2n))^{\eta-1}},
\end{equation}
but for infinitely many $n$,
\begin{equation} \label{eq114}
\frac{N_0(x_1^{2n})}{2} - N_0(x_1^n) > \frac{-2n}{(\log_2(2n))^{\eta-\eps}}.
\end{equation}
By Lemma~\ref{lem-seq}(ii), applied to $\{N_0(x_1^n)\}_{n\ge 1}$, (\ref{eq114}) certainly holds for $\eps=\eta-1$.

We claim that 
\begin{equation} \label{eq115}
X_G\setminus (B\cup L) \subset \bigcup_{\eta\in (1,2)} \bigcup_{c>0} \bigcup_{\eps\in (0,2-\eta)} \Lam(\eta,\eps,c).
\end{equation}
Indeed, for $x\in X_G\setminus B$ let $\eta^*$ be the infimum of $\eta$ for which (\ref{eq112}) holds for $n$ sufficiently large (note that
$x\not\in B$ means such $\eta$ exists). Then $\eta^*\in [1,2)$ by Lemma~\ref{lem-seq}(ii),
and (\ref{eq112}) holds with $\eta = \eta^* + \frac{2-\eta^*}{3}$ for  $n$ sufficiently large, whereas (\ref{eq114}) holds for 
$\eps\in (\eta-\eta^*,2-\eta)=(\frac{2-\eta^*}{3},\frac{2(2-\eta^*)}{3})$. 
Let
$$
\gam(n):= N_0^*(x_1^{n}),\ \ n\ge 1.
$$
It is clear tthat
$$
|\gam(n+1)-\gam(n)| \le 2,\ \ n\ge 1.
$$
It follows from (\ref{eq112}) and Lemma~\ref{lem-exp} that
$$
\gam(n) - \frac{\gam(2n)}{2} \ge \frac{2n}{(\log_2(2n))^\eta}-C(\log_2 n)^2 \ge \frac{n}{(\log_2(2n))^\eta}
$$
for $n$ sufficiently large. Thus, the sequence $\{\gam(n)\}_{n\ge 1}$ satisfies the assumptions of Lemma~\ref{lem-seq}(i).
By Lemma~\ref{lem-seq}(i),
either there exists $c>0$ such that for all  $n$ sufficiently large
$$
N_0^*(x_1^{2n})\ge c\,\frac{2n}{(\log_2(2n))^{\eta-1}}\,,
$$
which together with the above yields that $x\in \Lam(\eta,\eps,c)$, or else there exists $\eps>0$
such that for infinitely many $n$,
$$
N_0^*(x_1^{2n})\le -\eps n\ \ \ \mbox{and}\ \ \ N^*_0(x_1^n)-\frac{N^*_0(x_1^{2n})}{2}\le \frac{n}{\log_2(2n)}\,,
$$
hence by  Lemma~\ref{lem-exp},
$$
N_0(x_1^n) - \frac{N_0(x_1^{2n})}{2} \le \frac{n}{\log_2(2n)} + c(\log_2 n)^2 \le \frac{2n}{\log_2(2n)}
$$
for infinitely many $n$, so that $x\in L$,
proving the claim.
Since the union in (\ref{eq115}) can be taken over rational $\eta,c,\eps$, it suffices to show that $\Hk^s(\Lam(\eta,\eps,c))
=0$ for $\eps\in (0,2-\eta)$.

Let $\Gam_{n}(\eta,\eps,c)$ be the collection of words $u$ of length $2n$ for which (\ref{eq112}), (\ref{eq113}), and (\ref{eq114}) hold
(as before, this is well defined). If $u\in \Gam_{n}(\eta,\eps,c)$, then by (\ref{equ1}) and (\ref{eq114})
$$
\log_2(\Pmu[u]\,2^{2ns}) \ge \frac{-2ns}{(\log_2(2n))^{\eta-\eps}},
$$
hence
\begin{equation} \label{eq116}
2^{-2ns} \le \exp\Bigl(\frac{2ns(\ln 2)}{\log_2(2n)^{\eta-\eps}}\Bigr) \Pmu[u].
\end{equation}
By the definition of $\Gam_{n}(\eta,\eps,c)$ and Lemma~\ref{lem-ldev2}, with $t=\frac{2c}{(\log_2(2n))^{\eta-1}}$,
\begin{eqnarray*}
\sum_{u\in \Gam_{n}(\eta,\eps,c)} \Pmu[u] & \le & \Pmu\Bigl(x:\ N_1^*(x_1^{2n})\ge c\frac{2n}{(\log_2(2n))^{\eta-1}}\Bigr) \\
                                          & \le & c_2\exp\Bigl(-\frac{\wtil{c}n}{(\log_2(2n))^{2\eta-2}}\Bigr),
\end{eqnarray*}
with $\wtil{c}=4c_3c^2$. Combining this with (\ref{eq116}) yields
$$
\sum_{u\in \Gam_{n}(\eta,\eps,c)} 2^{-2ns} \le \exp c_2\Bigl[\frac{2ns(\ln 2)}{\log_2(2n)^{\eta-\eps}} - \frac{\wtil{c}n}{(\log_2(2n))^{2\eta-2}}\Bigr].
$$
Recall that $\eps<2-\eta$, and therefore $\eta-\eps> 2\eta-2$ and the right-hand side of the last inequality is summable in $n$. It follows that,
by taking $n_1$ sufficiently large, the sum
$$
\sum_{n\ge n_1} \sum_{u\in \Gam_{n}(\eta,\eps,c)} 2^{-2ns}
$$
can be made arbitrarily small. Since for any $n_1$, the union
$$
\bigcup_{n\ge n_1} \bigcup_{u\in \Gam_{n}(\eta,\eps,c)} [u]
$$
covers $\Lam(\eta,\eps,c)$, this implies $\Hk^s(\Lam(\eta,\eps,c))=0$ (and hence $\Hk^{\psi_\theta}(\Lam(\eta,\eps,c))=0$), completing the proof.
\qed\end{proof}

\medskip

\noindent
{\bf Acknowledgement.}
The research of B. S. was supported in part by the NSF grant DMS-0968879. He is grateful to the Microsoft Research Theory Group
for hospitality during 2010-2011. He would also like to thank the organizers of the conference ``Fractals and Related Fields II'' for
the excellent meeting and stimulating atmosphere. The authors are grateful to the referee for careful reading of the manuscript and many helpful
comments.
%
%
%

%
%

\begin{thebibliography}{99.}%
%
%
\bibitem{Bedf} Bedford, T.: Crinkly curves, Markov partitions
and box dimension in self-similar sets. Ph.D.\ Thesis, University
of Warwick, 1984.
\bibitem{Falconer} Falconer, K.: {\em Fractal Geometry. Mathematical Foundations and Applications.} (Wiley, Chichester, 1990).
\bibitem{Fan} Fan, A., Liao, L., Ma, J.: Level sets of multiple ergodic averages. Preprint (2011) arXiv:1105.3032. To appear in Monatsh.\ Math.
(published online 05 November 2011).
\bibitem{Hoeffding} Hoeffding, W.: Probability inequalities for sums of bounded random variables. J.\ Amer.\ Statist.\ Assoc. {\bf 58}, 
no.\ 302, 13--30 (1963).
\bibitem{KPS1} Kenyon, R., Peres, Y., Solomyak, B.: Hausdorff dimension of the multiplicative golden mean shift.
{C.\ R.\ Math.\ Acad.\ Sci.\ Paris} {\bf 349}, 625--628 (2011)
\bibitem{KPS2} Kenyon, R., Peres, Y., Solomyak, B.: Hausdorff dimension for fractals invariant under the
multiplicative integers. Preprint (2011) arXiv 1102.5136. To appear in Ergodic Th.\ Dynam.\ Sys.
\bibitem{LG} Lalley S., Gatzouras, D.: Hausdorff and box dimensions of certain self-affine fractals.
Indiana Univ.\ Math.\ J. {\bf 41}, no.\ 2, 533--568 (1992).
\bibitem{McMullen} McMullen, C.: The Hausdorff dimension of general
Sierpinski carpets. Nagoya Math.\ J. {\bf 96}, 1--9 (1984).
\bibitem{Peres} Peres, Y.: The self-affine carpets of McMullen and Bedford have infinite Hausdorff measure.
{Math.\ Proc.\ Camb.\ Phil.\ Soc.} {\bf 116}, 513--526 (1994).
\bibitem{Rogers} Rogers, C.A.: {\em Hausdorff measures.} (Cambridge University Press, Cambridge, 1970).
\bibitem{RT} Rogers, C.A., Taylor, S.J.: Functions continuous and singular with respect to a Hausdorff measure.  Mathematika
 {\bf 8}, 1--31 (1961).

\end{thebibliography}
%

\end{document}